\documentclass[11pt,reqno]{amsart} 
\usepackage[utf8]{inputenc}
\usepackage{amsfonts,amsmath,amsthm,amssymb}
\usepackage[justification=centering]{caption}
\usepackage{color}
\usepackage{fullpage}
\usepackage{enumitem}   
\usepackage{hyperref,float}
\usepackage{mathtools}
\usepackage{xcolor}
\usepackage{thm-restate}
\usepackage{xurl}
\usepackage{subcaption}
\usepackage{thmtools}
\usepackage{hyperref}
\usepackage{cleveref}			
\usepackage{dsfont}
\usepackage{enumitem}
\usepackage{tikz}
\newtheorem{theorem}{Theorem}[section]
\newtheorem{lemma}[theorem]{Lemma}
\newtheorem{corollary}[theorem]{Corollary}

\theoremstyle{definition}
\newtheorem{example}[theorem]{Example}  
\newtheorem{definition}[theorem]{Definition}

\newcommand{\Asc}{\mathrm{Asc}}

\newcommand{\Des}{\mathrm{Des}}

\newcommand{\Sym}{\mathfrak{S}}

\DeclareMathOperator\supp{supp}

\title{Boolean intervals in the weak Bruhat order of a\\finite Coxeter group}

\author[Adenbaum]{Ben Adenbaum}
\address[B.~Adenbaum]{Dartmouth College}
\email{\textcolor{blue}{\href{mailto:Benjamin.M.Adenbaum.GR@dartmouth.edu}{Benjamin.M.Adenbaum.GR@dartmouth.edu}}}

\author[Elder]{Jennifer Elder}
\address[J.~Elder]{Department of Computer Science, Mathematics and Physics, Missouri Western State University}
\email{\textcolor{blue}{\href{mailto:jelder8@missouriwestern.edu}{jelder8@missouriwestern.edu}}}

\author[Harris]{Pamela E. Harris}
\address[P.~E.~Harris]{Department of Mathematical Sciences, University of Wisconsin-Milwaukee}
\email{\textcolor{blue}{\href{mailto:peharris@uwm.edu}{peharris@uwm.edu}}}

\author[Mart\'inez Mori]{J.~Carlos Mart\'inez Mori}
\address[J.~C.~Mart\'inez Mori]{H. Milton Stewart School of Industrial and Systems Engineering, Georgia Institute of Technology}
\email{\textcolor{blue}{\href{mailto:jcmm@gatech.edu}{jcmm@gatech.edu}}}

 
\begin{document}
\begin{abstract}
Given a Coxeter group $W$ with Coxeter system $(W,S)$, where $S$ is finite.  
We provide a complete characterization of Boolean intervals in the weak order of $W$ uniformly for all Coxeter groups in terms of independent sets of the Coxeter graph.  
Moreover, we establish that the number of Boolean intervals of rank $k$ in the weak order of  $W$ is 
    ${i_k(\Gamma_W)\cdot|W|}\,/\,2^{k}$,
    where $\Gamma_W$ is the Coxeter graph of $W$ and $i_k(\Gamma_W)$ is the number of independent sets of size $k$ of $\Gamma_W$ when $W$ is finite.
Specializing to $A_n$, we recover the characterizations and enumerations of Boolean intervals in the weak order of $A_n$ given in \cite{elder2023Boolean}. 
We provide the analogous results for types $C_n$ and $D_n$, including the related generating functions and additional connections to well-known integer sequences. 

\end{abstract}

\maketitle

\section{Introduction}
\label{sec: introduction}

Given a poset $P$, it is a classical problem to determine if (and if so, how many times) a particular poset $T$ appears as an interval (subposet) within $P$. 
In this work, we answer this fully when $P$ is the weak order of a Coxeter group $W$ with generating set $S$ and $T$ is a Boolean poset. 
The \textit{(right) weak order} of $W$ is a poset whose element set is $W$ and whose cover relations arise from the (right-hand side) application of a generator $s \in S$. 
Recall that a poset is said to be \emph{Boolean} if it is isomorphic to the poset of subsets of a set $I$, where the cover relations arise from set inclusion.
If $|I| = k < \infty$, then a Boolean poset is ranked.
Refer to Section~\ref{sec: results for coxeter groups} for technical definitions and notation, which follow the conventions of Bjorner and Brenti in~\cite{bjorner2006combinatorics}.

Let $\ell(w)$ denote the \emph{length} of $w\in W$ (i.e., the minimal number of generators needed to express $w$).
For $w\in W$, let 
\begin{equation*}
    \Des(w) \coloneqq \{s\in S\, |\, \ell(ws) = \ell(w) - 1\}
\end{equation*}
denote the \emph{(right) descent set} of $w$.
A generator $s \in S$ is said to be in the \emph{support} of $w$, denoted $\supp(w)$, if it appears in a minimal decomposition of $w$. 
A result of Tenner~\cite[Corollary 4.4]{Tenner_Intervals} states that, in the weak order of a Coxeter group $W$, an
interval $[v,w]$ is Boolean if and only if $v^{-1}w$ is a product of commuting generators.
The following is a restatement of this result using the language of descent sets; this version plays a key role in our main results.
\begin{theorem}
\label{thm: supp}
If $[v,w]$ is a Boolean interval in the weak order of $W$, then $\supp(v^{-1}w)\cap \Des(v) = \emptyset$.
Furthermore, if $J = \{s_{i_1}, \ldots , s_{i_k}\}$ consists of commuting generators and $J \cap \Des(v) = \emptyset$, then $[v, vs_{i_1}\cdots s_{i_k}]$ is a Boolean interval of rank $k$.
\end{theorem}

Recall that given a graph $G = (V, E)$, a subset $S \subseteq V$ of its vertices is an \emph{independent set} if no two vertices $u, v \in S$ are adjacent in $E$. 
We use Theorem~\ref{thm: supp} to establish a bijection between Boolean intervals in the weak order of a Coxeter group $W$ and the collection of independent sets of an induced subgraph of its Coxeter graph $\Gamma_W$; refer to Definition~\ref{def:Coxeter graph}. 
\begin{theorem}
\label{thm:local_bools_above}
For any $v \in W$, the set of $w \geq v$ such that $[v,w]$ is a Boolean interval is in bijection with the collection of independent sets of $\Gamma_W[S \setminus \Des(v)]$; the induced subgraph of $\Gamma_W$ obtained by deleting all of the elements in $\Des(v)$ from its vertex set.
\end{theorem}

Let $i_k(G)$ denote the number of independent sets of $G$ of size $k$.
When $W$ is finite, the bijection in Theorem~\ref{thm:local_bools_above} implies the following.
\begin{corollary}
\label{cor:global_bools_below}
There are
\begin{equation*}
    \frac{i_k(\Gamma_W)\cdot |W|}{2^k}
\end{equation*}
Boolean intervals of rank $k$ 
in the weak order of the Coxeter group $W$.
\end{corollary}

The remainder of this paper is organized as follows.
In Section~\ref{sec: results for coxeter groups} we provide some necessary background and prove our main results: Theorems \ref{thm: supp} and \ref{thm:local_bools_above}, and Corollary \ref{cor:global_bools_below}.
In Section~\ref{sec: specializing to classical coxeter groups}, we apply our main results to the Coxeter groups of types $A_n$, $C_n$, and $D_n$.
First, we specialize Theorem~\ref{thm:local_bools_above} and Corollary~\ref{cor:global_bools_below} to the symmetric group $\Sym_{n}$ and recover the following:
\begin{enumerate}
    \item 
    The Boolean intervals with a generator as minimal element are enumerated by products of at most two Fibonacci numbers~\cite[Proposition~5.9]{Tenner_Intervals}.
    \item 
    The Boolean intervals with minimal element $v$ are enumerated by a product of Fibonacci numbers, whose indices correspond to block sizes of a partition of the ascent set of $v$~\cite[Theorem 1.1]{elder2023Boolean}.
\end{enumerate}
Then, for Coxeter groups of types $C_n$ and $D_n$, we prove that the number of Boolean intervals (of all ranks) with a fixed minimal element involves products of Fibonacci numbers and that, for type $D_n$, it also involves a Fibonacci-like sequence with initial values $1$ and $4$ \cite[\href{https://oeis.org/A000285}{A000285}]{OEIS}; refer to Theorems~\ref{thm: total C_n} and \ref{thm:Booleans type D}, respectively. 
We also give formulas for the number of Boolean intervals of rank $k$ in their weak orders; refer to Theorem~\ref{thm: k C_n} for the type $C_n$ result and Theorem~\ref{thm:Booleans in Dn} for the type $D_n$ result. 

We conclude by remarking that in Theorem~\ref{thm:local_bools_above} we did not assume that $W$ is finite, or even irreducible. 
Hence, in Section~\ref{sec: specializing to affine coxeter groups}, we give formulas for the number of Boolean intervals with any given minimal element for all of the infinite families of affine irreducible Coxeter groups, whose counts also involve products of Fibonacci numbers. We conclude the manuscript with some directions for future research.

\section{Results for Coxeter groups}
\label{sec: results for coxeter groups}

In this section, we introduce some necessary background and notation, and provide proofs for Theorems~\ref{thm: supp} and ~\ref{thm:local_bools_above}, as stated in the introduction.

\subsection{Notation}
\label{sec: notation}

Let $W$ be a Coxeter group with Coxeter system $(W,S)$.
Let $m$ be its associated \textit{Coxeter matrix}, of dimensions $|S|\times|S|$, whose entries are defined by 
\begin{align*}
    m(s,s')&=m(s',s)\\
    m(s,s')&=1 \mbox{ if and only if } s=s'.
\end{align*}

Associated to a Coxeter system is an undirected labeled graph known as the \emph{Coxeter graph}.
\begin{definition}
\label{def:Coxeter graph}
Given a Coxeter system $(W,S)$, its \textit{Coxeter graph} $\Gamma_W$ has vertex set $S$ and an edge between $s,s'\in S$ if and only if $m(s,s') \geq 3$. 
By convention, edges are labeled with their corresponding weight $m(s,s')$ only when $m(s,s') > 3$.
\end{definition}
Note that two generators $s, s' \in S$ commute if and only if they are not adjacent in $\Gamma_W$.

\subsection{Proofs of Main Results}
\label{sec: proofs of main results}

We are now ready to prove our first main result.
\begin{proof}[Proof of Theorem~\ref{thm: supp}]
We first show that if $[v,w]$ is a Boolean interval, then $\supp(v^{-1}w)\cap \Des(v) = \emptyset$.
Suppose that $[v,w]$ is a Boolean interval of rank $k$, which implies $\ell(w) = \ell(v) + k$.
Now, suppose by way of contradiction that $\supp(v^{-1}w) \cap \Des(v) \neq \emptyset$, and let $s \in \supp(v^{-1}w) \cap \Des(v)$.
Since $s \in \Des(v)$, there is a reduced word for $v$ that ends with $s$.
Furthermore, since $s \in \supp(v^{-1}w)$, and since \cite[Corollary 4.4]{Tenner_Intervals} implies $v^{-1}w$ is a product of $k$ commuting generators, there is a reduced word for $v^{-1}w$ that begins with $s$.
Note that a reduced word for $v$ followed by a reduced word for $v^{-1}w$ is a word for $w$.
In particular, we concatenate the aforementioned reduced words for $v$ and $v^{-1}w$ to obtain a word for $w$ with two consecutive instances of $s$, which in turn implies there is a word for $w$ of length $\ell(v)+ k - 2$.
However, this contradicts the fact that $\ell(w)=\ell(v)+k$.

We now show that if $J = \{s_{i_1}, \ldots , s_{i_k}\}$ consists of commuting generators and $J \cap \Des(v) = \emptyset$, then $[v, vs_{i_1}\cdots s_{i_k}]$ is a Boolean interval of rank $k$.
Note that for any distinct $s, s' \in J$, we have $s \notin \Des(vs')$ as under the action of $W$ on the associated simple roots, the root associated to $s$, $\alpha_s$, is fixed by the transformation associated to $s'$ as they commute. 
Consequently $vs'(\alpha_s) = v(\alpha_s)$. 
 As $v(\alpha_s)$ is a positive root, since $s\notin \Des(v)$, then so is $vs'(\alpha_s)$, which implies that $s\notin \Des(vs')$. Thus, $[v,vs_{i_1}s_{i_2}\cdots s_{i_k}]$ is a Boolean interval of rank $k$.

\end{proof}

We now use Theorem~\ref{thm: supp} to prove our second main result.
\begin{proof}[Proof of Theorem~\ref{thm:local_bools_above}]
Note that, by the definition of $\Gamma_W$, a subset of $k$ commuting generators in $S \setminus \Des(v)$ is a subset of $k$ vertices of $\Gamma_W[S\setminus \Des(v)]$ with no pairwise adjacency (i.e., an independent set), and vice versa.
Now, consider any $w \in W$ such that $[v,w]$ is a Boolean interval of rank $k$.
Then, the first part of Theorem~\ref{thm: supp} implies $\supp(v^{-1}w) \cap \Des(v) = \emptyset$, so that $\supp(v^{-1}w) \subseteq S\setminus \Des(v)$.
In particular, $\supp(v^{-1}w)$ is an independent set of $\Gamma_W[S\setminus \Des(v)]$ of size $k$.
Conversely, suppose $J = \{s_{i_1}, \ldots , s_{i_k}\} \subseteq S \setminus \Des(v)$ is an independent set of $\Gamma_W[S\setminus \Des(v)]$ of size $k$.
Then, the second part of Theorem~\ref{thm: supp} implies $[v, vs_{i_1}\cdots s_{i_k}]$ is a Boolean interval of rank $k$.
This establishes the bijection.

\end{proof}
	
We now turn to the enumeration of Boolean intervals of a given rank $k$ in the weak order of a Coxeter group $W$, as stated in Corollary~\ref{cor:global_bools_below}. 
Before proving the result, recall that if $J \subseteq S$, then $W_J$ is the parabolic subgroup generated by $J$.
Analogously, $W^J$ is defined as
\begin{equation*}
    W^J \coloneqq \{w \in W \; | \; \Des(w) \subseteq S \setminus J\}.
\end{equation*}
It is known (refer to~\cite[Corollary 2.4.5]{bjorner2006combinatorics}) that $W^J$ is a system of minimal left coset representatives for $W/W_J$. 
If $W$ is finite, this implies 
\begin{equation*}
    |W^J| = \frac{|W|}{|W_J|}.
\end{equation*}
 
Returning to our problem, for each subset $J \subseteq S$ of commuting generators, we count the number of elements $v \in W$ for which there is a Boolean interval with minimal element $v$ supported on $J$.
We are now ready to prove our enumerative result.
\begin{proof}[Proof of Corollary \ref{cor:global_bools_below}]
Let $J = \{s_{i_1},s_{i_1}, \ldots, s_{i_k}\} \subseteq S$ consist of any $k$ commuting generators and let $u =s_{i_1} s_{i_2} \cdots s_{i_k}$.
By Theorem~\ref{thm: supp}, for each element $v \in W$ with $\Des(v) \cap J = \emptyset$, there is a unique element $w \in W$ with $w \geq v$ and interval $[v,w] \simeq [e,u]$.
Note that $W^J$ is the set of elements $v \in W$ for which $\Des(v)\cap J = \emptyset$, and that $|W_J|=2^{k}$ since the generators in $J$ commute.
Therefore, there are 
\begin{equation*}
    |W^J| = \frac{|W|}{|W_J|} = \frac{|W|}{2^{k}}
\end{equation*}
such intervals.
Lastly, recall that $J$ is any independent set of $\Gamma_W$ of size $k$.
In particular, this count is the same for any such independent set, and each distinct choice of independent set defines distinct intervals.
Therefore, we immediately derive that there are
\begin{equation*}
    \frac{i_k(\Gamma_W) \cdot |W|}{2^k}
\end{equation*}
Boolean intervals of rank $k$, as claimed.

\end{proof}

\section{Specializing to Classical Coxeter Groups}
\label{sec: specializing to classical coxeter groups}

We now specialize the results in Section~\ref{sec: results for coxeter groups} to the weak order of Coxeter groups of type $A_n$, $C_n$, and $D_n$.

\subsection{Boolean Intervals in the Weak Order of $A_n$}
\label{sec: A_n}

For $n \geq 1$, the Coxeter group of type $A_n$ is described as follows. 
Let $S = \{s_1, s_2, \ldots, s_n\}$ with 
\begin{equation*}
    m(s_i,s_j) 
    = 
    \begin{cases}
        1 & \text{ if } i = j \\
        3 & \text{ if } |i - j| = 1 \\
        2 & \text{ if } |i - j| > 1.
	\end{cases}
\end{equation*}
Note that $A_n$ is isomorphic to $\mathfrak{S}_{n+1}$. 
As illustrated in Figure~\ref{fig:An}, this system has a path graph $P_n$ on $n$ vertices as its Coxeter graph $\Gamma_{A_n}$.
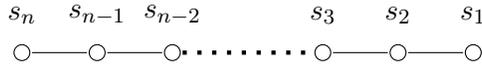
\begin{figure}[ht]
	\centering
	\begin{tikzpicture}
		\node (1) at (-3, 0) {};
		\node (2) at (-2, 0) {};
		\node (3) at (-1, 0) {};
		\node (4) at (1, 0) {};
		\node (5) at (2, 0) {};
		\node (6) at (3, 0) {};
		\node (l1) at (-3,.5) {$s_n$};
		\node (l2) at (-2,.5) {$s_{n-1}$};
		\node (l3) at (-1,.5) {$s_{n-2}$};
		\node (l4) at (1, .5) {$s_3$};
		\node (l5) at (2, .5) {$s_2$};
		\node (l6) at (3, .5) {$s_1$};
		\draw[-] (1) -- (2);
		\draw[-] (2) -- (3);
		\draw[loosely dotted, ultra thick] (3) -- (4);
		\draw[-] (4) -- (5);
		\draw[-] (5) -- (6);
		\draw[color=black] (1) circle (3pt);
        \draw[color=black] (2) circle (3pt);
        \draw[color=black] (3) circle (3pt);
        \draw[color=black] (4) circle (3pt);
        \draw[color=black] (5) circle (3pt);
        \draw[color=black] (6) circle (3pt);
	\end{tikzpicture}    
	\caption{
        The graph $\Gamma_{A_n}$ associated to $A_n$.
    }
	\label{fig:An}
\end{figure}
	
We begin with the following illustrative example.
\begin{example}
	Consider the Coxeter group of type $A_8 $ which is isomorphic to $\mathfrak{S}_9$ and let $\pi = 513649728\in \mathfrak{S}_9$ be expressed in one-line notation.
    Then, $\Des(\pi) = \{1,4,6,7\}$ and $[8] \setminus \Des(\pi) = \{2, 3, 5, 8\}$.
    The induced subgraph $\Gamma_{A_8}[\{2, 3, 5, 8\}]$ is depicted in Figure~\ref{fig:An_sub}.
	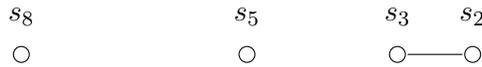
\begin{figure}[ht]
	   \centering
    	\begin{tikzpicture}
    		\node (2) at (-2, 0) {};
    		\node (3) at (-3, 0) {};
    		\node (5) at (-5, 0) {};
    		\node (8) at (-8,  0) {};
    		\draw[-] (2) -- (3);
    		\node (l2) at (-2, .5) {$s_2$};
    		\node (l3) at (-3, .5) {$s_3$};
    		\node (l5) at (-5, .5) {$s_5$};
    		\node (l8) at (-8, .5) {$s_8$};
    		\draw[color = black ] (5) circle (3pt);
            \draw[color = black ] (2) circle (3pt);
            \draw[color = black ] (3) circle (3pt);
            \draw[color = black ] (8) circle (3pt);
    	\end{tikzpicture}
     	\caption{The induced subgraph $\Gamma_{A_8}[\{2, 3, 5, 8\}]$.}
      \label{fig:An_sub}
	\end{figure}
 
	There are twelve independent sets of $\Gamma_{A_8}[\{2, 3, 5, 8\}]$, namely:
    \begin{equation*}
        \emptyset,\{s_2\},\{s_3\},\{s_5\},\{s_8\},\{s_2, s_5\},\{s_2, s_8\},\{s_3, s_5\},\{s_3, s_8\}, \{s_5, s_8\},\{s_2, s_5, s_8\}, \text{ and } \{s_3, s_5, s_8\}.
    \end{equation*}
    As listed in Table~\ref{tab:An_ex}, each independent set $I$ corresponds to a Boolean interval of rank $|I|$ with minimal element $\pi$.
    \begin{table}[ht]
    \centering
    \begin{tabular}{|cc|cc|cc|cc|}\hline
    \multicolumn{2}{|c|}{Rank $0$} & \multicolumn{2}{|c|}{Rank $1$} & \multicolumn{2}{|c|}{Rank $2$} & \multicolumn{2}{|c|}{Rank $3$}      \\ \hline\hline
    $I$ & $w$ & $I$       & $w$    & $I$            & $w$    & $I$ & $w$ \\ \hline
    $\emptyset$      & $\pi$      & $\{s_2\}$    & $531649728$    & $\{s_2, s_5\}$  & $531694728$ & $\{s_2, s_5, s_8\}$ & $531694782$ \\
                     &            & $\{s_3\}$    & $516349728$    & $\{s_2, s_8\}$  & $531649782$ & $\{s_3, s_5, s_8\}$ & $516394782$ \\
        &          & $\{s_5\}$ & $513694728$ & $\{s_3, s_5\}$ & $516394728$ &     &          \\
        &          & $\{s_8\}$ & $513649782$ & $\{s_3, s_8\}$ & $516349782$ &     &          \\
        &          &           &             & $\{s_5, s_8\}$ & $513694782$ &     &          \\ \hline
    \end{tabular}
    \caption{
        Boolean intervals with minimal element $\pi$ organized by rank, independent set $I$, and corresponding maximal element $w$.
    }
    \label{tab:An_ex}
    \end{table}
\end{example}

Define the Fibonacci numbers by $F_{n+2}=F_{n+1}+F_{n}$, with initial values $F_1=F_2=1$ as in \cite[\href{https://oeis.org/A000045}{A000045}]{OEIS}.
The number of independent sets of the path graph $P_n$ on $n$ vertices is known to be the Fibonacci number $F_{n+2}$, and the number of independent sets of a graph is the product of the number of independent sets of its connected components.
Therefore, the total number of Boolean intervals (of all possible ranks) above an element of $A_n$ is a product of Fibonacci numbers.
This specialization of Theorem~\ref{thm:local_bools_above} recovers \cite[Proposition 5.9]{Tenner_Intervals} of Tenner and \cite[Theorem 1.1]{elder2023Boolean} of
Elder, Harris, Kretschmann, and Mart\'inez Mori.
We formally state it below.

\begin{theorem}{\cite[Theorem 1.1]{elder2023Boolean}}
\label{thm:An_rephrased}
    Let $\pi = \pi_1 \pi_2 \cdots \pi_{n+1} \in A_{n}$ be in one-line notation.
    Let $\Gamma_{A_n}[S \setminus \Des(\pi)]$ be the subgraph of $\Gamma_{A_n}$ induced by deleting the elements in $\Des(\pi)$ from its vertex set.
    Partition the vertex set of $\Gamma_{A_n}[S \setminus \Des(\pi)]$ into connected components $b_1, b_2, \ldots, b_k$.
    Then, the number of Boolean intervals $[\pi, w]$ in the weak order of the Coxeter group of type $A_n$ with fixed minimal element $\pi$ and arbitrary maximal element $w$ (including the case $\pi=w$) is given by
    \begin{equation*}
        \prod_{i=1}^kF_{|b_i|+2},
    \end{equation*}
    where $F_{\ell}$ is the $\ell$th Fibonacci number and $F_1=F_2 = 1$. 
\end{theorem}

The number of independent sets of size $k$ of $P_n$ is also known.
\begin{lemma}{\cite[Proposition 1.1.iv]{hopkins1984some}}
\label{lem:i(P_n)}
    Let $n \geq 0$ and $0 \leq k \leq n$.
    Then, $i_k(P_n) = \binom{n+1-k}{k}$. 
\end{lemma}

Together with Lemma~\ref{lem:i(P_n)}, the specialization of Corollary~\ref{cor:global_bools_below} to $A_n$ recovers \cite[Theorem 1.3]{elder2023Boolean}.
\begin{corollary}{\cite[Theorem 1.3]{elder2023Boolean}}
There are 
\begin{equation}
\label{eq: f(n,k)}
    \frac{n!}{2^k}\binom{n-k}{k}
\end{equation} 
Boolean intervals of rank $k$ in the weak order of the Coxeter group of type $A_{n-1}$.
\end{corollary}

\subsection{Boolean Intervals in the Weak Order of $C_n$}
\label{sec: C_n}

For $n \geq 2$, the Coxeter group of type $C_n$ is described as follows. 
Let $S = \{ s_1, s_2, \dots , s_{n}\}$ with
\begin{equation*}
    m(s_i,s_j) 
    = \begin{cases}
		1 & \text{ if } i = j \\
		3 & \text{ if } |i - j| = 1 \text{ and } i,j < n\\
		2 & \text{ if } |i - j| > 1\\
		4 &	\text{ if } i + j = 2n - 1.
	\end{cases}
\end{equation*}
As illustrated in Figure~\ref{fig:Cn}, this system has a path graph $P_n$ on $n$ vertices as its Coxeter graph $\Gamma_{C_n}$.
\begin{figure}[ht]
    \centering
    \begin{tikzpicture}
		\node (0) at (-4, 0) {};
		\node (1) at (-3, 0) {};
		\node (2) at (-2, 0) {};
		\node (3) at (-1, 0) {};
		\node (4) at (1, 0) {};
		\node (5) at (2, 0) {};
		\node (6) at (3, 0) {};
		\node (l0) at (-4, .5) {$s_n$};
		\node (l0e) at (-3.5, .25) {4};
		\node (l1) at (-3,.5) {$s_{n-1}$};
		\node (l2) at (-2,.5) {$s_{n-2}$};
		\node (l3) at (-1,.5) {$s_{n-3}$};
		\node (l4) at (1, .5) {$s_{3}$};
		\node (l5) at (2, .5) {$s_{2}$};
		\node (l6) at (3, .5) {$s_{1}$};
		\draw[-] (0) -- (1);
		\draw[-] (1) -- (2);
		\draw[-] (2) -- (3);
		\draw[loosely dotted, ultra thick] (3) -- (4);
		\draw[-] (4) -- (5);
		\draw[-] (5) -- (6);
		
		\draw[color = black] (0) circle (3pt);
		\draw[color = black] (1) circle (3pt);
        \draw[color = black] (2) circle (3pt);
        \draw[color = black] (3) circle (3pt);
        \draw[color = black] (4) circle (3pt);
        \draw[color = black] (5) circle (3pt);
        \draw[color = black] (6) circle (3pt);
	\end{tikzpicture}
    \caption{
        The graph $\Gamma_{C_n}$ associated to $C_n$.
    }
	\label{fig:Cn}
\end{figure}
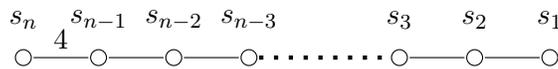

Now, consider the isomorphic representation of the Coxeter group of type $C_n$ in $\mathfrak{S}_{2n}$ where, if $\pi=\pi_1\pi_2\cdots\pi_n\in \mathfrak{S}_{2n}$, we impose the additional requirement that 
\begin{equation}
\label{eq: C_n conditions}
    \pi_i=k \text{ if and only if } \pi_{2n - i + 1}= 2n - k + 1.
\end{equation}
In this way, we treat the elements of $C_n$ as \textit{mirrored permutations} of $[2n]$ in one-line notation~\cite{10.2140/involve.2017.10.263}.
The cover relations defining the weak order of $C_n$ are as follows:
\begin{equation}
     \sigma \lessdot \tau 
     \text{ if and only if } 
     \begin{cases} 
         \sigma s_n = \tau &\text{ when } n \in \Des(\tau)\\
         \sigma s_i s_{2n-i+1}=\tau  &\text{ when } i \in \Des(\tau)\cap [n-1].
     \end{cases}
\end{equation}
Figure~\ref{fig:Weak of C3} depicts the Hasse diagram of the weak order of the Coxeter group of type $C_3$. 
\begin{figure}[ht]
    \centering
\includegraphics[width=.9\linewidth]{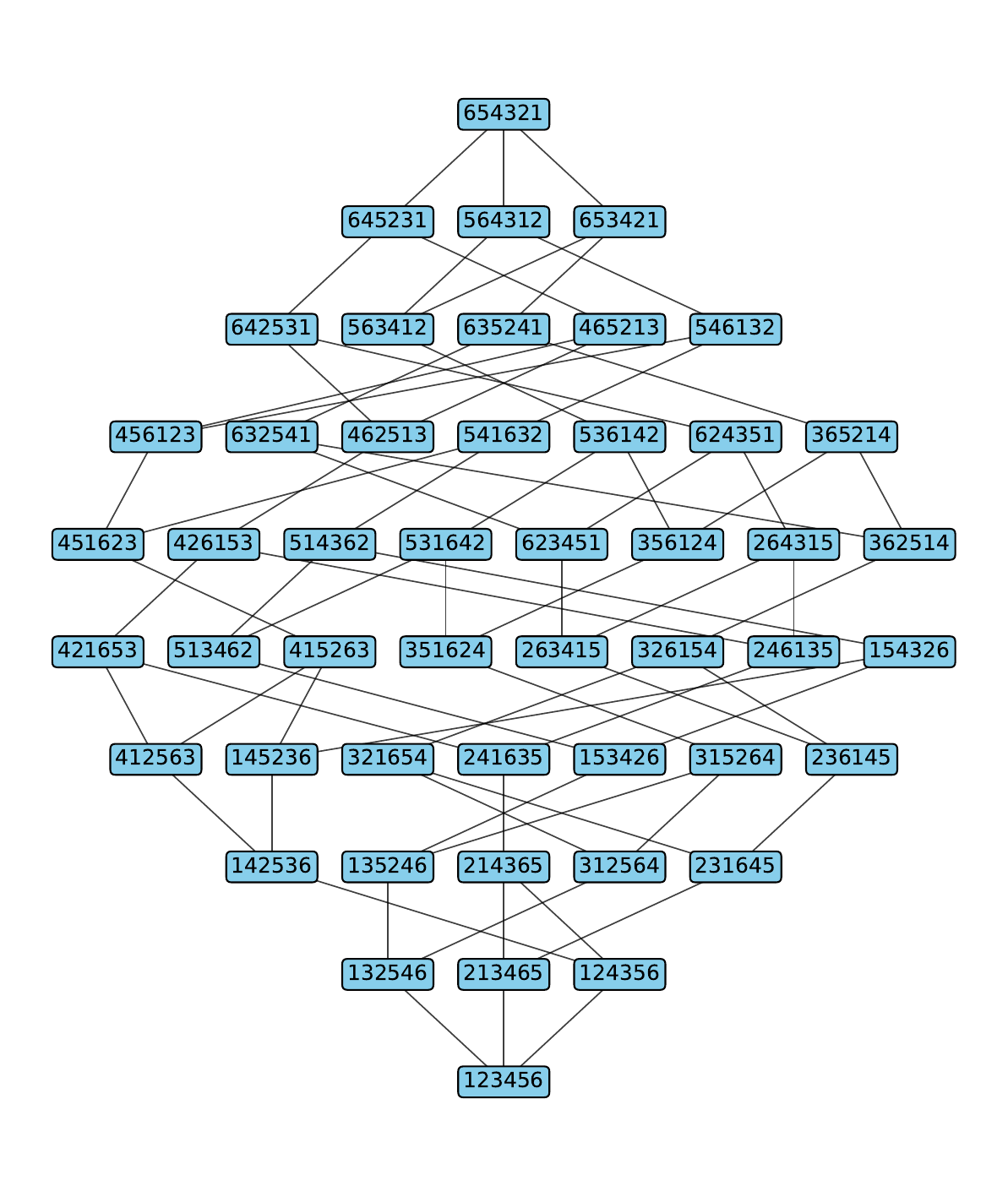}
    \caption{
        Weak order of the Coxeter group of type $C_3$ represented as mirrored permutations in one-line notation.
    }
    \label{fig:Weak of C3}
\end{figure}

Consider the following illustrative example.
\begin{example}
    Consider the Boolean interval $B_3$ in the weak order of $\mathfrak{S}_6$ with minimal element $\pi=451623$ and maximal element $w=546132$, as depicted in Figure~\ref{fig: cube}.
    
    \begin{figure}
        \centering
        \begin{subfigure}{.495\linewidth}
          \centering
          \includegraphics[width=.8\linewidth]{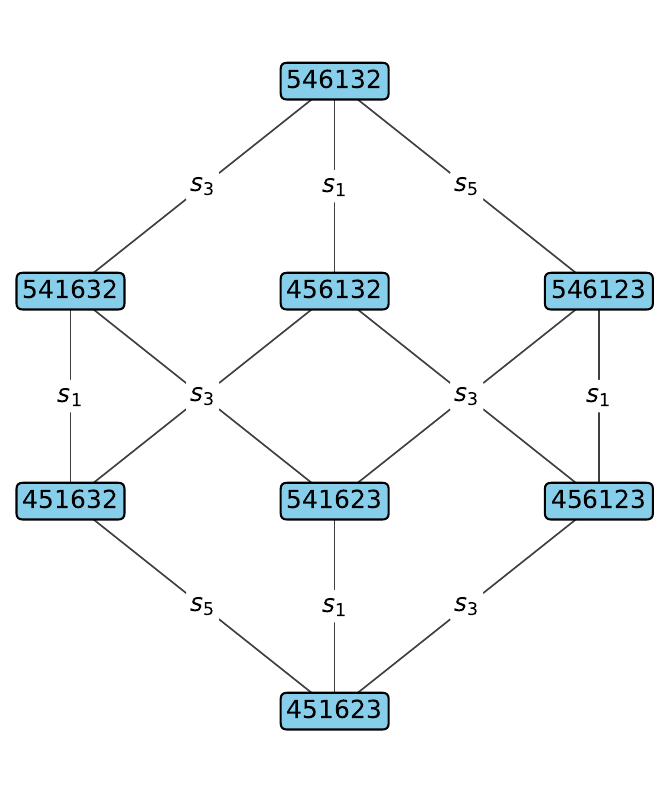}
          \caption{A Boolean interval of rank three in the weak order of $\mathfrak{S}_6$ with two edges, corresponding to $s_3$, which are kept in the weak order of $C_3$.}
          \label{fig: cube}
        \end{subfigure}%
        \begin{subfigure}{.495\linewidth}
          \centering
          \includegraphics[width=.8\linewidth]{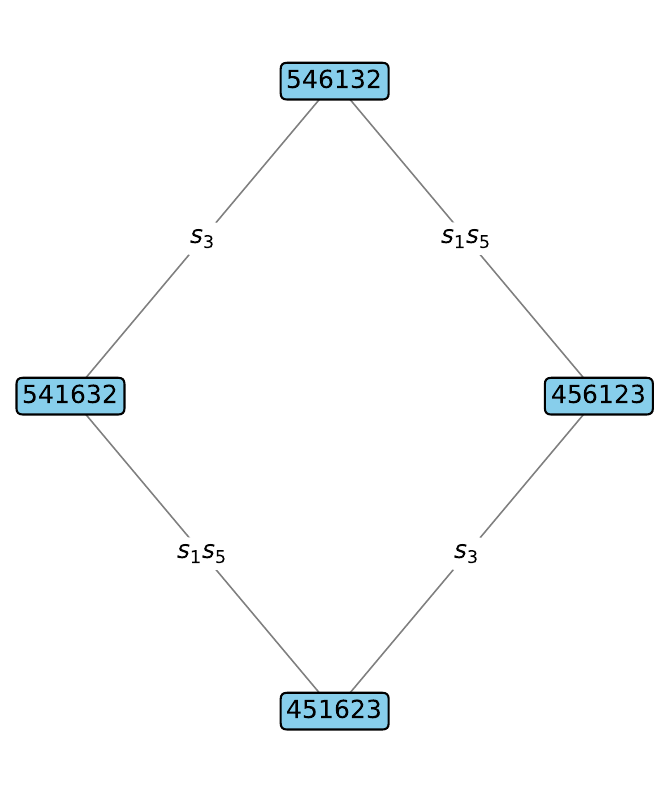}
          \caption{A Boolean interval of rank two in weak order of $C_3$.}
          \label{fig: square}
        \end{subfigure}
        \caption{Examples of Boolean intervals in the weak order of $C_3$.}
        \label{fig: test}
    \end{figure}

    Note that both $\pi$ and $w$ are elements of $C_3$, as they satisfy the condition given in~\eqref{eq:Cn condition} (i.e., they are mirrored permutations).
    Moreover, note that $\Des(w)=\{1, 3, 5\}$.
    Consider the application of $s_1$, $s_3$, and $s_5$:
    \begin{itemize}
        \item
        Applying $s_3$ to $\pi$ and $w$, respectively, gives another element of $C_3$.
        Hence, $\pi $ is covered by $\pi s_3$ and $w s_3$ is covered by $w$, and there are edges from $451623$ to $456123$ and from $541632$ to $546132$ in the weak order of $C_3$.
        \item 
        Applying any single one of $s_1$ or $s_5$ to $\pi$ yields permutations that are not in $C_3$, as they do not satisfy the condition given in~\eqref{eq:Cn condition} (i.e., they are not mirrored permutations).
        However, applying $s_1s_5$ to $\pi$ gives another element of $C_3$.
        Hence, $\pi$ is covered by $\pi s_1s_5 = 541632$ and there is an edge from $451623$ to $541632$ in the weak order of $C_3$.
        \item 
        Applying any single one of $s_1$ or $s_5$ to $\pi s_3=456123$ yields permutations that are not in $C_3$, as they do not satisfy the condition given in~\eqref{eq:Cn condition}.
        However, applying $s_1s_5$ to $\pi s_3$ gives another element of $C_3$.
        Hence, $\pi s_3$ is covered by $\pi s_3s_1s_5 = 546132$ and there is an edge from $456123$ to $546132$ in the weak order of $C_3$.
    \end{itemize}
    The edges described in this example form a Boolean interval of rank two in the weak order of $C_3$ with minimal element $451623$ and maximal element $546132$, as depicted in Figure~\ref{fig: square}.
\end{example}

Again, the total number of Boolean intervals (of all possible ranks) above an element of $C_n$ is a product of Fibonacci numbers. 
\begin{theorem}
\label{thm: total C_n}
Let $\pi = \pi_1\pi_2\cdots\pi_{2n}\in\mathfrak{S}_{2n}$ satisfy $\pi_i=k$ if and only if $\pi_{2n-i+1}=2n-k+1$ for all $i\in[n]$. 
Let $\Gamma_{C_n}[S \setminus \Des(\pi)]$ be the subgraph of $\Gamma_{C_n}$ induced by deleting the elements in $\Des(\pi)$ from its vertex set.
Partition the vertex set of $\Gamma_{C_n}[S \setminus \Des(\pi)]$ into connected components $b_1, b_2, \ldots, b_k$.
Then, the number of Boolean intervals $[\pi, w]$ in the weak order of the Coxeter group of type $C_n$ with fixed minimal element $\pi$ and arbitrary maximal element $w$ (including the case $\pi=w$) is given by
\begin{equation*}
    \prod_{i=1}^kF_{|b_i|+2},
\end{equation*}
where $F_{\ell}$ is the $\ell$th Fibonacci number and $F_1=F_2 = 1$.
\end{theorem}
\begin{proof}
Conditions~\eqref{eq: C_n conditions} imply that, for $v \in C_n$ with associated mirrored permutation $\pi$ and $1 \leq i \leq n$, $s_i \in \Des(v)$ if and only if $i \in \Des(\pi)$. 
Therefore, the elements of $\Gamma_{C_n}[S\setminus \Des(v)]$ correspond to the $1 \leq i \leq n$ such that $i \in \Asc_{C_n}(\pi)$. 
This means that the connected components of $\Gamma_{C_n}[S\setminus \Des(v)]$ are exactly the maximal blocks of consecutive entries, call them $b_1, b_2, \dots, b_k$.
Since each $b_i$ corresponds to a path graph of size $|b_i|$, by Theorem~\ref{thm:local_bools_above}, we have that the number of Boolean intervals above $v$ (equivalently above $\pi$) is $\prod_{i=1}^kF_{|b_i|+2}$, as desired.
\end{proof}

\begin{example}
Let $n = 9$ and suppose $\pi \in \Sym_{18}$ satisfies $\Asc_{C_n}(\pi)=\{1,3,4,9\}$.
Then, there are $F_{3} \cdot F_{4} \cdot F_{3}=12$ Boolean intervals with minimal element $\pi$.
For an explicit example, consider
    \begin{equation*}
        \pi = 3 (17) 4 7 (18) (14) (11) 9  6 (13) (10) 8 5 1 (12) (15) 2 (16).  
    \end{equation*}
    Note that the cover relations are determined by left multiplying $\pi$ by $s_1s_{17}$, $s_3s_{15}$, $s_4s_{14}$, and $s_9$.
    There is one Boolean interval of rank 0, there are four Boolean intervals of rank 1, five Boolean intervals of rank 2, and two Boolean intervals of rank 3. 
    This gives a total of 12 Boolean intervals with minimal element $\pi$, as expected.
\end{example}

Given that $\Gamma_{C_n}$ is a path graph $P_n$ on $n$ vertices, our next result follows directly from Corollary~\ref{cor:global_bools_below} and Lemma~\ref{lem:i(P_n)}. 
\begin{theorem}
\label{thm: k C_n}
    There are 
    \begin{equation*}
        2^{n-k} n!\binom{n+1-k}{k}
    \end{equation*}
    Boolean intervals of rank $k$ in the weak order of the Coxeter group of type $C_n$.
\end{theorem}
Note that setting $k=1$ in Theorem~\ref{thm: k C_n} gives the number of edges in the weak order of $C_n$, namely the number of Boolean intervals of rank one, denoted $B_1$. For ease of reference we state this result below and remark that this sequence corresponds to \cite[\href{https://oeis.org/A014479}{A014479}]{OEIS}.
\begin{corollary}
    The number of edges in the weak order of $C_n$ is given by 
    $\frac{n}{2}\cdot 2^nn!$.
\end{corollary}

Note that setting $k=2$ in Theorem \ref{thm: k C_n} gives the sequence for the number of Boolean intervals or rank two in the weak order of $C_n$, which  for $n \geq 1$ begins with
\begin{equation*}
    0, 0, 12, 288, 5760, 115200, 2419200, 54190080, 1300561920, 33443020800, \ldots.
\end{equation*}
This sequence does not appear in the Online Encyclopedia of Integer Sequences~\cite{OEIS}.

We conclude this subsection with the following exponential generating function for the number of Boolean intervals of rank $k$ in the weak order of the Coxeter group of type $C_n$.

\begin{corollary}

    Let $f_C(n,k)$ count the number of Boolean intervals of rank $k$ in $C_n$. Then \[\sum_{n\ge 0}\sum_{k\ge 0} f_C(n,k) q^k \frac{x^n}{n!}= \frac{(2+q)x+2qx^2}{1-2x-2qx^2}+1.\] 
\end{corollary}

\begin{proof}
If $n\ge 3$, then  $f_C(n,k)$ satisfies the following recurrence: $f_C(n,k) = 2nf_C(n-1,k) + 2n(n-1)f_C(n-2, k-1)$ depending on whether $s_1$ is supported in the interval. So the desired generating function will be rational with denominator $1-2x-2qx^2$. Then note in $C_1$ there are two Boolean intervals of rank 0 and only one Boolean interval of rank $1$. In $C_2$ the number of Boolean intervals is eight of rank 0, and 8 of rank 1 as there are two intervals above the minimal element and one above every element that is not maximal. By then applying the recurrence we have that if $b$ is the coefficient of $\frac{x^2}{2}$, then $b+8+4q$ will correspond to $\sum_{i=0}^{2} f_C(2,i)q^i = 8+8q$. So the coefficient of $\frac{x^2}{2}$ is $4q$. So the numerator is $(2+q)x+2qx^2$, leading to the generating function of $\frac{(2+q)x+2qx^2}{1-2x-2qx^2}$. Then to account for the $n=0$ case we add 1. 
\end{proof}
\begin{corollary}
Let $f_C(n)$ count the number of Boolean intervals in $C_n$. 
Then,
\begin{equation*}
    \sum_{n \geq 0} f_C(n) \frac{x^n}{n!}= \frac{3x+2x^2}{1-2x-2x^2} + 1.
\end{equation*}
\end{corollary}

\subsection{Boolean Intervals in the Weak Order of $D_n$}

We now consider the Coxeter group of type $D_n$. 
Let $S = \{s_{n-1}, s_{n-1}',s_{n-2},\ldots, s_{1}\}$ with
\begin{equation*}
    m(s_i,s_j) 
    = 
    \begin{cases}
		1 & \text{ if }s_i=s_j\\
		3 & \text{ if } |i-j|=1\\
		2 & \text{ if } |i-j| > 1 \text{ or }i=j \text{ and }s_i\neq s_j.
	\end{cases}
\end{equation*}

In this case, $\Gamma_{D_n}$ is the graph on $n$ vertices as illustrated in Figure~\ref{fig:Dn}.
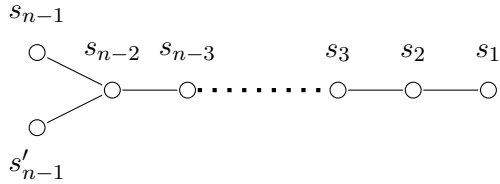
\begin{figure}[ht]
    \centering
\begin{tikzpicture}
    \node (1) at (-3, 0.5) {};
    \node (1') at (-3, -0.5) {};
    \node (2) at (-2, 0) {};
    \node (3) at (-1, 0) {};
    \node (4) at (1, 0) {};
    \node (5) at (2, 0) {};
    \node (6) at (3, 0) {};
    \node (l1) at (-3,1) {$s_{n-1}$};
        \node (l1) at (-3,-1) {$s_{n-1}'$};
    \node (l2) at (-2,.5) {$s_{n-2}$};
    \node (l3) at (-1,.5) {$s_{n-3}$};
    \node (l4) at (1, .5) {$s_{3}$};
    \node (l5) at (2, .5) {$s_{2}$};
    \node (l6) at (3, .5) {$s_{1}$};
    \draw[-] (1) -- (2);
        \draw[-] (1') -- (2);
    \draw[-] (2) -- (3);
    \draw[loosely dotted, ultra thick] (3) -- (4);
    \draw[-] (4) -- (5);
    \draw[-] (5) -- (6);
    
\draw[color = black ] (1) circle (3pt);
    \draw[color = black ] (1') circle (3pt);
    \draw[color = black ] (2) circle (3pt);
    \draw[color = black ] (3) circle (3pt);
    \draw[color = black ] (4) circle (3pt);
    \draw[color = black] (5) circle (3pt);
    \draw[color = black] (6) circle (3pt);
\end{tikzpicture}    
    \caption{The graph $\Gamma_{D_n}$ associated to $D_n$.}
    \label{fig:Dn}
\end{figure}

We again return to the perspective of mirrored permutations.
Consider the isomorphic representation of the Coxeter group of type $D_n\subset C_n \subset \mathfrak{S}_{2n}$ where, if $\pi= \pi_1 \pi_2 \ldots \pi_n\pi_{n+1}\pi_{n+2}\ldots \pi_{2n} \in \mathfrak{S}_{2n}$ we impose the additional requirements that 
\begin{equation}
\label{eq:Cn condition}
    \pi_i=k \text{ if and only if } \pi_{2n - i + 1}= 2n - k + 1.
\end{equation}
and $\{\pi_1, \pi_2,\ldots , \pi_n\}$ always contains an even number of elements
from the set $\{n + 1, n + 2, . . . , 2n\}$.
In this way, we treat the elements of $D_n$ as mirrored permutations of $[2n]$ in one-line notation~\cite{10.2140/involve.2017.10.263}.
The cover relations defining the weak order of $D_n$ are as follows:

\begin{equation}
     \sigma \lessdot \tau 
     \text{ if and only if } 
     \begin{cases} 
         \sigma (s_ns_{n-1}s_{n+1}s_n) = \tau &\text{ when } n \in \Des(\tau)\\
         \sigma s_i s_{2n-i+1}=\tau  &\text{ when } i \in \Des(\tau)\cap [n-1].
     \end{cases}
\end{equation}
Figure~\ref{fig:Weak of D3} depicts the Hasse diagram of the weak order of the Coxeter group of type $D_3$. 
Compare Figures~\ref{fig:Weak of C3} and ~\ref{fig:Weak of D3} to verify that $D_n \subset C_n$.

\begin{figure}[h!]
    \centering
\includegraphics[width=.8\linewidth]{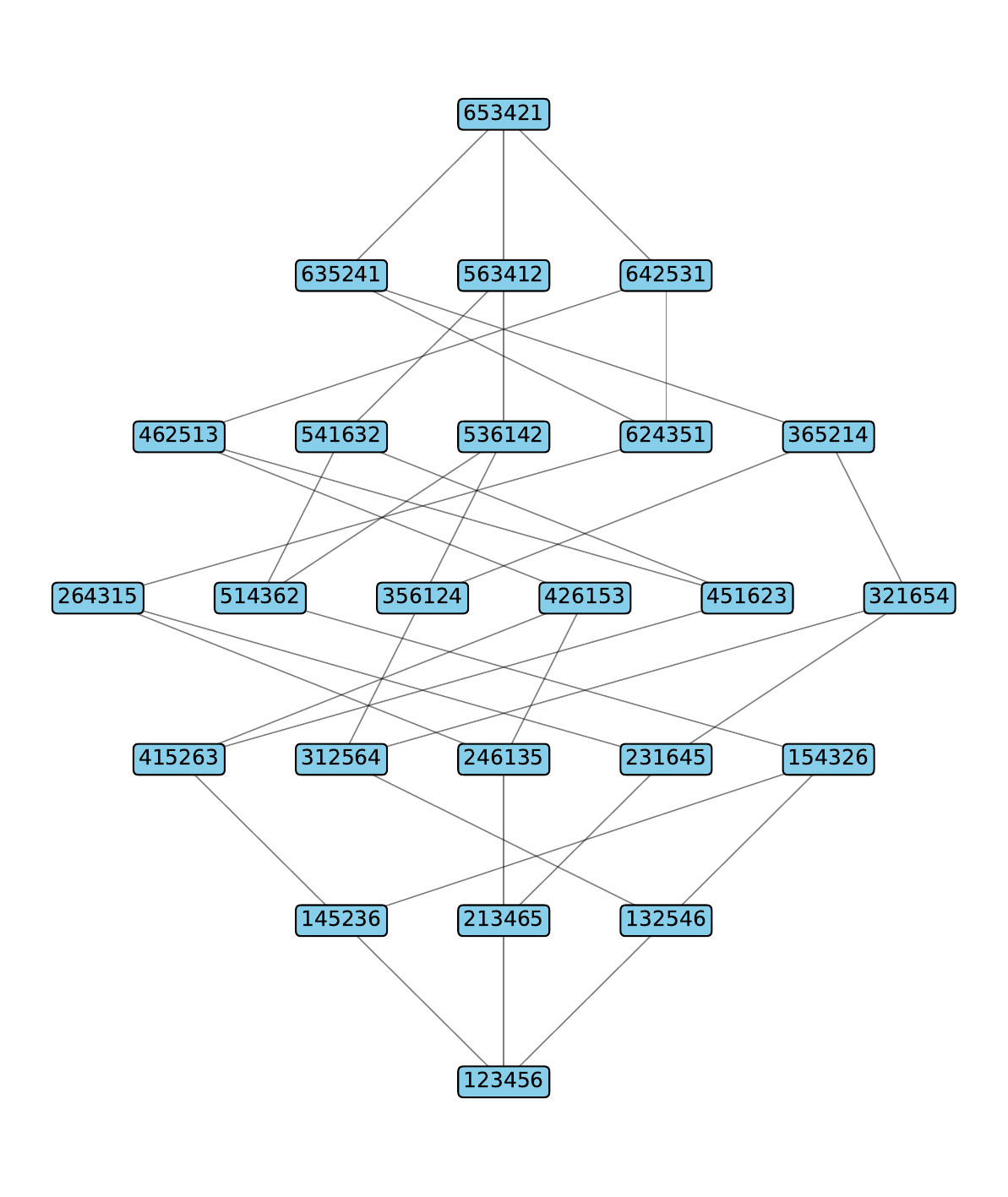}
    \caption{
        Weak order of $D_3$ represented as mirrored permutations in one-line notation.
    }
    \label{fig:Weak of D3}
\end{figure}

We begin by giving a recursive formula for the number of independent sets of $\Gamma_{D_n}$.
\begin{lemma}\label{lem:Type_D_independent_set_count_global}
    Let $(d_n)_{n\ge 1}$ be the sequence given by $d_n = d_{n-1}+d_{n-2}$ with $d_1=1 , d_2 = 4$. Then for $n\ge 1$, $d_n $ is the number of independent sets of $\Gamma_{D_n}$. 
\end{lemma}
\begin{proof}
    The initial conditions are such that $d_4$ and $d_5$ correspond to where $\Gamma_{D_n}$ has 4 and 5 vertices, which is a star with 3 leaves having 9 independent sets and adjoining a new vertex connected to one of the leafs, which has 14 independent sets. In general, the number of independent sets of $\Gamma_{D_n}$ for $n\ge 6$ can be enumerated via the case where the leaf connected to a vertex of degree 2 is in the set, and alternatively where the leaf connected to a vertex of degree 2 is not in the set. 
    If the leaf connected to a vertex of degree 2 is in the set, then the remaining elements are an independent set of the graph obtained by deleting the leaf and it's neighbor, which is $\Gamma_{D_{n-2}}$. On the other hand, if leaf connected to a vertex of degree 2 is not in the  set, then the independent set is of the graph obtained by just deleting the leaf, which is $\Gamma_{D_{n-1}}$.  
    Adding those cases yields the desired recurrence $d_{n}=d_{n-1}+d_{n-2}$.

    For $n=1,2,3$, note that $\Gamma_{D_1}$ is the empty graph, which has 1 independent set, $\Gamma_{D_2}$ is two disjoint vertices which has 4 independent sets, and $\Gamma_{D_3}=P_3$ which has 5 independent sets.
\end{proof}
The sequence $d_n$ of Lemma~\ref{lem:Type_D_independent_set_count_global} appears in~\cite[\href{https://oeis.org/A000285}{A000285}]{OEIS}, which already mentions the enumeration of independent sets in this family of graphs (unfortunately, no attribution is available). 
The same sequence has been considered before in the context of collapse free Hecke algebras in $D_n$~\cite{huang16hecke}.

\begin{theorem}
\label{thm:Booleans type D}
Let $\pi \in D_n$. 
Let $\Gamma_{D_n}[S \setminus \Des(\pi)]$ be the subgraph of $\Gamma_{D_n}$ induced by deleting the elements in $\Des(\pi)$ from its vertex set.
Partition the vertex set of $\Gamma_{D_n}[S \setminus \Des(\pi)]$ into connected components $b_1, b_2, \ldots, b_k$, and let $b_0$ be the (possibly empty) connected component containing a vertex of degree three.
Then, the number of Boolean intervals $[\pi, w]$ in the weak order of the Coxeter group of type $D_n$ with fixed minimal element $\pi$ and arbitrary maximal element $w$ (including the case $\pi=w$) is given by
\begin{equation*}
    d_{|b_0|}\cdot \prod_{i=1}^k F_{|b_i|+2},
\end{equation*}
where we define $d_0=1$, $F_{\ell}$ is the $\ell$th Fibonacci number, and $F_1=F_2 = 1$.
\end{theorem}
\begin{proof}
Note that $\Gamma_{D_n}[S\setminus \Des(\pi)]$ has a single, possibly empty, component containing a vertex of degree 3, and that any other connected component is a path graph. 
Then, by Theorem~\ref{thm:local_bools_above}, together with Lemma~\ref{lem:Type_D_independent_set_count_global}, and the fact that the number of independent sets of a path graph on $k$ vertices is $F_{k+2}$, the result follows.
\end{proof}

To count the number of Boolean intervals of a fixed rank, it suffices to count the number of independent sets of $\Gamma_{D_n}$ of size $k$ for all $n\geq 4$. We give this count next.

\begin{lemma}\label{lem:Type_D_independent_set_count_fixed_size}
    For $n\ge 4$, the number of independent sets of $\Gamma_{D_n}$ of size $k$ is given by
    \[i_k(\Gamma_{D_n}) = \binom{n-k}{k-2}+\binom{n-k-1}{k-1}+\binom{n-k}{k}.\]
\end{lemma}
\begin{proof}
    An independent set of $\Gamma_{D_n}$ of size $k$  must contain some subset of the two leaves adjacent to the degree 3 vertex. 
    If it contains both, then the remaining independent set is an independent set of size $k-2$ on the graph obtained by deleting these two leaves and the degree 3 vertex. 
    This deleting gives a path graph on $n-3$ vertices, so by Lemma~\ref{lem:i(P_n)}, the number of independent sets of size $k$ is $\binom{n-k}{k-2}$. 
    If only one of these two vertices is in the independent set, then the remaining elements are again an independent set on the path graph obtained by deleting the degree 3 vertex and the two adjacent leafs. 
    As there are two choices for which of these vertices is to be in the independent set, there are $2\binom{n-k-1}{k-1}$ such independent sets. 
    Finally, if neither of these two vertices are in the independent set, then the independent set is just an independent set of size $k$ on the graph obtained by deleting the two leaves, which is a path graph on $n-2$ vertices. 
    In this case, the number of such independent sets is given by Lemma~\ref{lem:i(P_n)} as  $\binom{n-k-1}{k}$. The claim then follows from the binomial identity $\binom{n-k-1}{k-1}+\binom{n-k-1}{k}=\binom{n-k}{k}$.
\end{proof}

Specializing Corollary \ref{cor:global_bools_below} to the Coxeter group of type $D_n$ and using Lemma \ref{lem:Type_D_independent_set_count_fixed_size} yields the following result.

\begin{theorem}
\label{thm:Booleans in Dn}
    There are \[2^{n-k-1}n! \bigg( \binom{n-k}{k-2}+\binom{n-k-1}{k-1}+\binom{n-k}{k}\bigg)\]
     Boolean intervals of rank $k$ in the weak order of the Coxeter group of type $D_n$.
\end{theorem}

Note that setting $k=1$ in Theorem~\ref{thm:Booleans in Dn} gives the number of edges in the weak order of $D_n$, namely the number of Boolean intervals of rank one, denoted $B_1$. For ease of reference we state this result below and remark that this sequence corresponds to 
\cite[\href{https://oeis.org/A019999}{A019999}]{OEIS}, which has been studied in the context of local bisection refinement for $n$-simplicial grids
              generated by reflection \cite{Maubach}.

\begin{corollary}
    The number of edges in the weak order of $D_n$ is given by $2^{n-2} n! n$.
\end{corollary}

Note that setting $k=2$ in Theorem \ref{thm:Booleans in Dn} gives the sequence for the number of Boolean intervals or rank two in the weak order of $D_n$, which  for $n \geq 1$ begins with
\begin{equation*}
    0, 1, 6, 144, 2880, 57600, 1209600, 27095040, 650280960, 16721510400, 459841536000, \ldots.
\end{equation*}
This sequence does not appear in the Online Encyclopedia of Integer Sequences~\cite{OEIS}.

We conclude this section with an exponential generating function for the number of Boolean intervals of rank $k$ in the weak order of $D_n$.
\begin{corollary}

    Let $f_D(n,k)$ count the number of Boolean intervals of rank $k$ in $D_n$. Then \[\sum_{n\ge 0}\sum_{k\ge 0} f_D(n,k) q^k \frac{x^n}{n!}= \frac{1}{2}\left(\frac{2x+(4q+q^2)x^2}{1-2x-2qx^2}\right)+1.\] 
\end{corollary}

\begin{proof}
If $n\ge 3$, then $f_D(n,k)$ satisfies the following recurrence: $f_D(n,k) = 2nf_D(n-1,k) + 2n(n-1)f_D(n-2, k-1)$ depending on whether $s_1$ is supported in the interval. So the desired generating function will be rational with denominator $1-2x-2qx^2$. Then note that in $D_1$ there is only one Boolean interval of rank $0$. In $D_2$ the number of Boolean intervals of rank 4 is zero, of rank 1 is four, and of rank 2 is one. By then applying the recurrence we have that if $b$ is the coefficient of $\frac{x^2}{2}$, then $b+4$ will correspond to $\sum_{i=0}^{2} f_D(2,i)q^i = 4+4q+q^2$. So the coefficient of $\frac{x^2}{2}$ is $4q+q^2$. So the numerator is $x+\frac{(4q+q^2)x^2}{2}$, leading to the generating function $\frac{x+(4q+q^2)\frac{x^2}{2}}{1-2x-2qx^2} =\frac{1}{2}\left(\frac{2x+(4q+q^2)x^2}{1-2x-2qx^2}\right)$. Then to account for the $n=0$ case we add 1. 
\end{proof}
\begin{corollary}
Let $f_D(n)$ count the number of Boolean intervals in $D_n$. 
Then,
\begin{equation*}
    \sum_{n \geq 0} f_D(n) \frac{x^n}{n!}= \frac{1}{2}\left(\frac{2x+5x^2}{1-2x-2x^2}\right) + 1.
\end{equation*}
\end{corollary}

\section{Specializing to Affine Coxeter Groups}
\label{sec: specializing to affine coxeter groups}
In this section, we briefly consider the enumeration of Boolean intervals in the weak order of the Coxeter groups $\tilde{A}_n, \tilde{B}_n, \tilde{C}_n,$ and $\tilde{D}_n$. As Corollary~\ref{cor:global_bools_below} only applies in the case where $W$ is finite, we will only be able to apply Theorem~\ref{thm:local_bools_above} to obtain results for the number of Boolean intervals above a given element. One thing we mention with regards to the affine types is that in type $\tilde{X}_n$, $|S|=n+1$, differing from the finite case where in $X_n$, $|S|=n$. We use the definitions of these groups via their diagrams from \cite[Section 2.5]{humphreys1992reflection}. This differs from previous sections where we additionally provided combinatorial embeddings of the Coxeter groups as subgroups of a permutation group, but as the arguments for these groups will not differ we omit them. Again in contrast to previous sections we do not include any further discussion about generating functions counting the number of Boolean intervals by rank in the affine irreducible types as in these cases there are an infinite number of Boolean intervals of a given rank with at most one exception.

We begin with type $\tilde{A}_n$, for $n\ge 2$.  In type $\tilde{A}_n$, the relations are exactly those of $A_{n}$, except that there is a new generator $s_0$ with $m(s_0,s_1)=m(s_0,s_{n}) = 3$ as well. 
This means that $\Gamma_{\tilde{A}_n}$ is a cycle graph with $n+1$ vertices. 

\begin{corollary}
      Suppose $v\in \tilde{A}_n$. If $v\neq e$, let $b_1, b_2, \dots, b_r$ be the connected components of $\Gamma_{\tilde{A}_n}[S\setminus\Des(v)]$. Then the number of Boolean intervals of the form $[v, w]$ with fixed minimal element $v$ and arbitrary maximal element $w$ (including the case $v=w$) is counted by \[\prod_{i=1}^r F_{|b_i|+2}.\] 
If $v=e$, then the number of Boolean intervals of the form $[e,w]$ is the $(n+1)$th Lucas number, see~\cite[p .46]{Comtet}. 
\end{corollary}
\begin{proof}
If $v\neq e$, then $\Des(v)\neq \emptyset$. Consequently $\Gamma_{\tilde{A}_n}[S\setminus\Des(v)]$ is a disjoint union of $r$ path graphs, with the sizes of the connected components $b_1, b_2, \dots, b_r$. Then by Theorem~\ref{thm:local_bools_above} and the fact that the number of independent sets of a path graph on $k$ vertices is $F_{k+2}$ the claim follows. 

Instead if $v = e$, then by Theorem~\ref{thm:local_bools_above} the number of Boolean intervals above $v$ is the number of independent sets of a cycle on $n+1$ vertices, which is the $(n+1)$th Lucas number 
\cite[\href{https://oeis.org/A000032}{A000032}]{OEIS}.
\end{proof}
\begin{figure}[htb]
	    \centering
	\begin{tikzpicture}
        \node (0) at (0,-1) {};
		\node (1) at (-3, 0) {};
		\node (2) at (-2, 0) {};
		\node (3) at (-1, 0) {};
		\node (4) at (1, 0) {};
		\node (5) at (2, 0) {};
		\node (6) at (3, 0) {};
		\node (l1) at (-3,.5) {$s_1$};
            \node (l0) at (0,-.5) {$s_0$};
		\node (l2) at (-2,.5) {$s_2$};
		\node (l3) at (-1,.5) {$s_3$};
		\node (l4) at (1, .5) {$s_{n-2}$};
		\node (l5) at (2, .5) {$s_{n-1}$};
		\node (l6) at (3, .5) {$s_n$};
		\draw[-] (1) -- (2);
		\draw[-] (2) -- (3);
		\draw[loosely dotted, ultra thick] (3) -- (4);
		\draw[-] (4) -- (5);
		\draw[-] (5) -- (6);
            \draw[-] (0) -- (1);
            \draw[-] (0) -- (6);
		
		\draw[color = black ] (1) circle (3pt);
        \draw[color = black ] (2) circle (3pt);
        \draw[color = black ] (3) circle (3pt);
        \draw[color = black ] (4) circle (3pt);
        \draw[color = black] (5) circle (3pt);
        \draw[color = black] (6) circle (3pt);
        \draw[color = black] (0) circle (3pt);
        
	\end{tikzpicture}    
	    \caption{The graph $\Gamma$ associated to $\tilde{A}_{n}$.}
	    \label{fig:AnA}
	\end{figure}
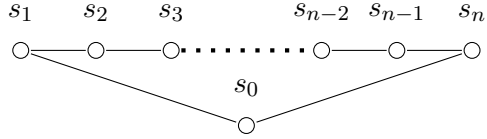
For type $\tilde{B}_n$, as $\Gamma_{\tilde{B}_n}$ is the same graph when edge labels are omitted as $\Gamma_{D_{n+1}}$, we have the following result via a proof identical to that of Theorem~\ref{thm:Booleans in Dn}.
\begin{corollary}
     Suppose $v\in \tilde{B}_n$. Let $b_0, b_1, b_2, \dots, b_r$ be the connected components of $\Gamma_{\tilde{B}_n}[S\setminus\Des(v)]$ with $b_0$ the, potentially empty, connected component containing a degree 3 vertex. Then the number of Boolean intervals of the form $[v, w]$ with fixed minimal element $v$ and arbitrary maximal element $w$ (including the case $v=w$) is counted by \[d_{|b_0|}\prod_{i=1}^r F_{|b_i|+2},\] with $d_0=1$.
\end{corollary}
\begin{figure}[htb]
	    \centering
	\begin{tikzpicture}
		\node (1) at (-3, 0.5) {};
            \node (1') at (-3, -0.5) {};
		\node (2) at (-2, 0) {};
		\node (3) at (-1, 0) {};
		\node (4) at (1, 0) {};
		\node (5) at (2, 0) {};
		\node (6) at (3, 0) {};
            \node (7) at (4, 0) {};
		\node (l1) at (-3,1) {$s_0$};
            \node (l1) at (-3,-1) {$s_0'$};
		\node (l2) at (-2,.5) {$s_1$};
		\node (l3) at (-1,.5) {$s_2$};
		\node (l4) at (1, .5) {$s_{n-4}$};
		\node (l5) at (2, .5) {$s_{n-3}$};
		\node (l6) at (3, .5) {$s_{n-2}$};
            \node (l7) at (4, .5) {$s'$};
		\draw[-] (1) -- (2);
            \draw[-] (1') -- (2);
		\draw[-] (2) -- (3);
		\draw[loosely dotted, ultra thick] (3) -- (4);
		\draw[-] (4) -- (5);
		\draw[-] (5) -- (6);
            \draw[-] (7) -- (6);
            \node (l0e) at (3.5, .25) {4};
		
	\draw[color = black ] (1) circle (3pt);
        \draw[color = black ] (1') circle (3pt);
        \draw[color = black ] (2) circle (3pt);
        \draw[color = black ] (3) circle (3pt);
        \draw[color = black ] (4) circle (3pt);
        \draw[color = black] (5) circle (3pt);
        \draw[color = black] (6) circle (3pt);
        \draw[color = black] (7) circle (3pt);
        
	\end{tikzpicture}    
 	    \caption{The graph $\Gamma_{\tilde{B}_n}$ associated to $\tilde{B}_n$.}
	    \label{fig:BnA}
	\end{figure}
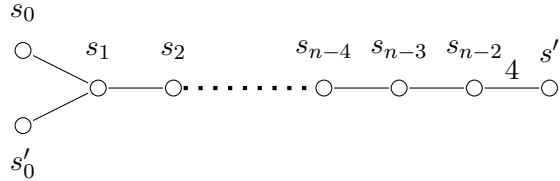
For Type $\tilde{C}_n$, as $\Gamma_{\Tilde{C}_n}$ is a path graph on $n+1$ vertices, we have the following result as a consequence of Theorem~\ref{thm:local_bools_above} and the fact that every induced subgraph of a path is a path. 

\begin{corollary}
         Suppose $v\in \tilde{C}_n$. Let $b_1, b_2, \dots, b_r$ be the connected components of $\Gamma_{\tilde{C}_n}[S\setminus\Des(v)]$. Then the number of Boolean intervals of the form $[v, w]$ with fixed minimal element $v$ and arbitrary maximal element $w$ (including the case $v=w$) is counted by \[\prod_{i=1}^r F_{|b_i|+2}.\]
\end{corollary}
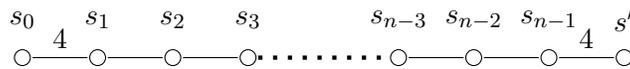
\begin{figure}[htb]
	    \centering
		\begin{tikzpicture}
		\node (0) at (-4, 0) {};
		\node (1) at (-3, 0) {};
		\node (2) at (-2, 0) {};
		\node (3) at (-1, 0) {};
		\node (4) at (1, 0) {};
		\node (5) at (2, 0) {};
		\node (6) at (3, 0) {};
		\node (l0) at (-4, .5) {$s_0$};
		\node (l0e) at (-3.5, .25) {4};
		\node (l1) at (-3,.5) {$s_1$};
		\node (l2) at (-2,.5) {$s_2$};
		\node (l3) at (-1,.5) {$s_3$};
		\node (l4) at (1, .5) {$s_{n-3}$};
		\node (l5) at (2, .5) {$s_{n-2}$};
		\node (l6) at (3, .5) {$s_{n-1}$};
		\draw[-] (0) -- (1);
		\draw[-] (1) -- (2);
		\draw[-] (2) -- (3);
		\draw[loosely dotted, ultra thick] (3) -- (4);
		\draw[-] (4) -- (5);
		\draw[-] (5) -- (6);
		\draw[-] (7) -- (6);
		\node (l7) at (4, .5) {$s'$};
		\node (l7e) at (3.5, .25) {4};
		\draw[color = black] (0) circle (3pt);
		\draw[color = black] (1) circle (3pt);
        \draw[color = black] (2) circle (3pt);
        \draw[color = black] (3) circle (3pt);
        \draw[color = black] (4) circle (3pt);
        \draw[color = black] (5) circle (3pt);
        \draw[color = black] (6) circle (3pt);
        \draw[color = black] (7) circle (3pt);
	\end{tikzpicture}
 	    \caption{The graph $\Gamma_{\tilde{C}_n}$ associated to $\tilde{C}_n$.}
	    \label{fig:CnA}
	\end{figure}
For type $\tilde{D}_n$, we have the following consequence of Theorem~\ref{thm:local_bools_above}
\begin{corollary}
          Suppose $v\in \tilde{D}_n$. If $v\neq e$, let $b_0, b_0', b_1, b_2, \dots, b_r$ be the connected components of $\Gamma_{\tilde{D}_n}[S\setminus\Des(v)]$, with $b_0,b_0'$ the, potentially empty, connected components containing a vertex of degree 3. Then the number of Boolean intervals of the form $[v, w]$ with fixed minimal element $v$ and arbitrary maximal element $w$ is (including the case $v=w$) counted by \[d_{|b_0|}d_{|b_0'|}\prod_{i=1}^r F_{|b_i|+2},\] where $d_0 = 1$.

          If $v = e$, then the number of Boolean intervals of the form $[v,w]$ is $d_n+2d_{n-2}$.
\end{corollary}
\begin{proof}
    When $v\neq e$, so $\Des(v)\neq \emptyset$, the first statement follows immediately from Theorem~\ref{thm:local_bools_above} together with the fact that any induced subgraph obtained by deleting at least one vertex of $\Gamma_{\tilde{D}_n}$ will have as it's connected components path graphs, together with at most two components that are copies of the Coxeter diagram for type $D$ for some smaller $k$'s. 

    For the second statement, the number of Boolean intervals above $e$ is just the number of independent sets of any size of $\Gamma_{\tilde{D}_n}$. If an independent set contains $s_{n-2}$, then the remaining elements are an independent set of the graph $\Gamma_{D_{n-2}}$ together with an isolated vertex $s'_{n-2}$, as $s_{n-3}$ cannot be in the set, of which there are $2d_{n-2}$ such independent sets by Lemma~\ref{lem:Type_D_independent_set_count_global}. In the case where $s_{n-2}$ is not in the set, the independent set is an independent set of the graph $\Gamma_{D_{n}}$ of which there are $d_n$ independent sets, again by Lemma~\ref{lem:Type_D_independent_set_count_global}.  
\end{proof}
 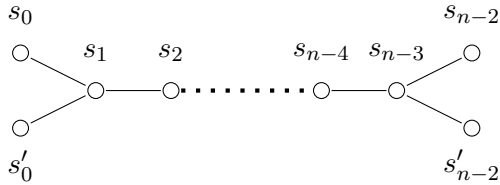
\begin{figure}[htb]
	    \centering
	\begin{tikzpicture}
		\node (1) at (-3, 0.5) {};
            \node (1') at (-3, -0.5) {};
		\node (2) at (-2, 0) {};
		\node (3) at (-1, 0) {};
		\node (4) at (1, 0) {};
		\node (5) at (2, 0) {};
		\node (6) at (3, .5) {};
            \node (6') at (3, -.5) {};
		\node (l1) at (-3,1) {$s_0$};
            \node (l1') at (-3,-1) {$s_0'$};
		\node (l2) at (-2,.5) {$s_1$};
		\node (l3) at (-1,.5) {$s_2$};
		\node (l4) at (1, .5) {$s_{n-4}$};
		\node (l5) at (2, .5) {$s_{n-3}$};
		\node (l6) at (3, 1) {$s_{n-2}$};
            \node (l6') at (3, -1) {$s'_{n-2}$};
		\draw[-] (1) -- (2);
            \draw[-] (1') -- (2);
		\draw[-] (2) -- (3);
		\draw[loosely dotted, ultra thick] (3) -- (4);
		\draw[-] (4) -- (5);
		\draw[-] (5) -- (6);
            \draw[-] (5) -- (6');
		
	\draw[color = black ] (1) circle (3pt);
        \draw[color = black ] (1') circle (3pt);
        \draw[color = black ] (2) circle (3pt);
        \draw[color = black ] (3) circle (3pt);
        \draw[color = black ] (4) circle (3pt);
        \draw[color = black] (5) circle (3pt);
        \draw[color = black] (6) circle (3pt);
        \draw[color = black] (6') circle (3pt);
	\end{tikzpicture}    
 	    \caption{The graph $\Gamma_{D_n}$ associated to $D_n$.}
	    \label{fig:DnA}
	\end{figure}

\section{Future Work}
In \cite{elder2023Boolean} the authors provided a parking function interpretation for the Boolean intervals of rank $k$ in the weak order of the Coxeter group of type $A_n$. We wonder if there is such a direct interpretation for Boolean intervals in other weak orders of classical Coxeter groups, which does not utilize the embedding of them into a larger permutation group. Moreover, such a combinatorial interpretation in terms of parking functions for the affine cases remains an open problem worthy of investigation. 
We also ask whether other combinatorial families of objects arise in characterizing and enumerating Boolean intervals in other well-known posets, including the strong Bruhat order of Coxeter groups and the Tamari lattice. 
For the reader interested in the latter, we point to the work of Fishel who enumerates chains in the Tamari lattice \cite{Susanna}.

\section*{Acknowledgments}
P.~E.~Harris was partially supported through a Karen Uhlenbeck EDGE Fellowship.
Part of this research was performed while J.~C. Mart\'inez Mori was visiting the Mathematical Sciences Research Institute (MSRI), now becoming the Simons Laufer Mathematical Sciences Institute (SLMath), which is supported by NSF Grant No. DMS-192893.
J.~C. Mart\'inez Mori is supported by Schmidt Science Fellows, in partnership with the Rhodes Trust.

\bibliographystyle{plain}
\bibliography{Bibliography}

\end{document}